 \documentclass{amsart}
\newtheorem{theorem}{Theorem}[section]

\newtheorem{lemma}[theorem]{Lemma}
\newtheorem{corollary}[theorem]{Corollary}

\newtheorem{remark}[theorem]{Remark}

\newtheorem{proposition}[theorem]{Proposition}

\newtheorem{definition}[theorem]{Definition}

\newcounter{rmnum}

\numberwithin{equation}{section}

\begin{document}
\title{Locally compact subgroup actions on  topological groups}
\author{Sergey A. Antonyan }

\address{Departamento de  Matem\'aticas,
Facultad de Ciencias, Universidad Nacional Aut\'onoma de M\'exico,
 04510 M\'exico Distrito Federal,  Mexico.}
\email{antonyan@unam.mx}

\begin{abstract}

Let $X$ be a Hausdorff topological group and $G$ a locally compact
subgroup of $X$. We show that  $X$ admits a locally finite $\sigma$-discrete
$G$-functionally open cover each member of which  is $G$-homeomorphic to  a  twisted product $G\times_H S_i$, where $H$ is a compact large subgroup of $G$  (i.e., the quotient $G/H$ is a manifold).  If, in addition,  the space of connected components of $G$ is compact and $X$ is normal, then $X$
itself is $G$-homeomorphic to a twisted product $G\times_KS$, where
$K$ is a maximal compact subgroup of $G$. This implies that $X$
 is $K$-homeomorphic to the product $G/K\times S$, and in particular,  $X$ is homeomorphic  to the  product $\Bbb R^n\times S$, where $n={\rm dim\,} G/K$.  Using these results we
prove  the inequality   $ {\rm dim}\, X\le {\rm dim}\, X/G + {\rm
dim}\, G$ for every Hausdorff   topological group $X$ and a locally compact subgroup $G$ of $X$.
\end{abstract}


\maketitle
\markboth{SERGEY A. ANTONYAN}{APPLICATIONS OF PROPER ACTIONS}

\section{Introduction}

By a $G$-space we mean a completely regular Hausdorff space together with a fixed continuous action of a given  Hausdorff topological  group  $G$ on it.

The notion of a proper $G$-space  was introduced in 1961 by R.~Palais \cite{pal:61} with the purpose to extend a substantial portion of the   theory of compact Lie group actions to the case of noncompact ones.

Recall that a  $G$-space   $X$ is called proper (in the sense of Palais \cite[Definition 1.2.2]{pal:61}), if each point of $X$ has a, so called, {\it small} neighborhood, i.e., a neighborhood $V$ such that for every point of $X$ there is a neighborhood $U$ with the property that the set
$\langle U,V\rangle=\{g\in G \ | \  gU\cap V\not= \emptyset\}$    has compact closure in $G$.

 Clearly, if $G$ is compact, then  every $G$-space is  proper.
Important examples of  proper $G$-spaces are the coset spaces
$G/H$ where  $H$ is  a compact subgroup of a locally compact group $G$.

In \cite{ant:10} we have shown that if $G$ is a locally compact subgroup of a  Hausdorff topological group $X$, then $X$ is a proper $G$-space with respect to  the natural action $g*x=xg^{-1}$ (or equivalently, $g*x=gx$) of $G$ on $X$. Based on the theory of proper actions, we  proved   in \cite{ant:10} that many topological properties (among which are normality and paracompactness)  are transferred from $X$ to its  quotient space $X/G$.

\medskip
This  paper is continuation of \cite{ant:10}; here we further apply the results and methods of the theory of proper actions in  the study of    topological  groups with respect to the natural  translation action of a  locally compact subgroup.

\smallskip

Below all topological groups are assumed to satisfy at least  the Hausdorff separation axiom.

All the notions involved in  the formulations of the following main results are defined in the next section:

\begin{theorem}\label{T:tub}  Let $X$ be a topological  group and   $G$ a locally compact subgroup of $X$. Then there exists a  compact subgroup $H\subset G$ such that:
\begin{enumerate}
\item  the quotient  $G/H$ is a manifold, and
\item  there exists a  locally finite $\sigma$-discrete cover of $X$ consisting of $G$-functionally open   sets $W_i$ such that
 the  closure $\overline{W_i}$ is a $G$-tubular set with the slicing subgroup $H$ (that is to say,   there exists a $G$-equivariant map $\overline{W_i}\to G/H$).
 \item each $\overline{W_i}$ is $G$-homeomorphic to a twisted product $G\times_H S_i$.
 \end{enumerate}
\end{theorem}

We recall that a locally compact group is called {\it almost connected} if its space of connected components is compact.

\begin{theorem}\label{T:tubmax}  Let $X$ be a topological  group,  $G$ an almost connected  subgroup of $X$, and $K$ a maximal compact subgroup of $G$. Then there exists a  locally finite $\sigma$-discrete cover of $X$ consisting of $G$-functionally open  sets  $W_i$ such that the  closure $\overline{W_i}$ is a $G$-tubular set with the slicing subgroup $K$ (that is to say,   there exists a $G$-equivariant map $\overline{W_i}\to G/K$). In particular, each $\overline{W_i}$ is homeomorphic to a   product $G/K\times S_i$.
\end{theorem}

\begin{theorem}\label{T:12}  Let $X$ be a normal topological  group,  $G$ an almost connected   subgroup of $X$, and $K$ a maximal compact subgroup of $G$. Then there exists a global $K$-slice $S$  in $X$ (equivalently,  there exists a $G$-equivariant map $X \to G/K$). In particular,  $X$ is homeomorphic to a   product $G/K\times S$.
\end{theorem}

Combining Theorem~\ref{T:12}  with a result of Abels~\cite[Theorem~2.1]{ab:74}, we obtain the following:

\begin{corollary}  Let   $X$ be a normal  group, $G$ an almost connected  subgroup of $X$, and  $K$  a maximal compact subgroup of $G$.    Then there exists a $K$-invariant subset  $S\subset X$ such that $X$ is $K$-homeomorphic to the product $G/K\times S$ endowed with the diagonal action of $K$ defined as follows: $k*(gK, s)=(kgK, ks)$. In particular, $X$ is homeomorphic to $\Bbb R^n\times S$, where  $n={\rm dim\,} G/K$.
\end{corollary}
\smallskip

\begin{corollary}\label{C:0}  Let $X$ be a normal topological  group,  $G$ an almost connected  subgroup of $X$, and $X/G$  the quotient space of all right cosets  $xG=\{xg\ | \ g\in G\}$, $x\in X$.
Then there exists a closed subset $S\subset X$ such that the restriction  $p|_S :S\to X/G$ is a perfect, open,  surjective  map.
\end{corollary}

This fact has the following two immediate corollaries  about transfer of properties from $X$ to $X/G$ and vice versa.

\begin{corollary}\label{C:3}  Let $\mathcal P$ be a topological property stable under open perfect   maps  and also inherited by closed subsets.  Assume that $X$ is a normal topological  group with the  property  $\mathcal P$ and let $G$ be an almost connected  subgroup of $X$. Then the quotient space  $X/G$ also has the property $\mathcal P$.
\end{corollary}

Among properties stable under open perfect   maps  and also inherited by closed subsets we highlight  strong paracompactness and realcompactness (see \cite[Exercises~5.3.C(c),  5.3H(d), and Theorem~3.11.4 and Exercises~3.11.G] {eng:77}).

\begin{corollary}\label{C:6}  Let $\mathcal P$ be a topological property that is both invariant  and inverse invariant under  open perfect  maps,  and also stable under multiplication  by a locally compact group.
 Assume that $X$ is a normal  topological  group and let $G$ be an almost connected  subgroup of $X$ such that  the quotient space $X/G$ has  the   property  $\mathcal P$. Then   the group $X$  also has the property $\mathcal P$.
\end{corollary}

Among such properties  we highlight   realcompactness (see
\cite[Theorem~3.11.14 and  Exercise~3.11.G, and also take into
account  that every locally compact group is realcompact]
{eng:77}).

Note that   A.\,V.~Arhangel'skii \cite{arh:05} has studied  properties which
are transferred from $X/G$ \ to $X$ for an arbitrary topological group $X$ and its locally compact subgroup $G$; see also \cite[Corollary~1.10]{ant:10}.

\medskip

Theorems \ref{T:tub} and \ref{T:12} are  further applied to prove the following Hurewicz type formula in dimension theory which, for a paracompact group $X$ and an almost connected subgroup $G\subset X$, was proved in \cite[Theorem~1.12]{ant:10}:

\begin{theorem}\label{T:13}  Let  $G$ be a locally compact subgroup of a    topological group   $X$.
   Then
$$ {\rm dim}\, X\le {\rm dim}\, X/G + {\rm dim}\, G.$$
\end{theorem}

\begin{remark}[\cite{pas:62}, \cite{skl:64}]\label{R} If in this theorem $X$ is a locally compact group then, in fact,  equality holds:
$$ {\rm dim}\, X = {\rm dim}\, X/G + {\rm dim}\, G.$$
\end{remark}

All the proofs are given in section~3.
\medskip

\section{Preliminaries}

Unless otherwise  stated, by  a {\it
group} we shall mean   a topological group $G$ satisfying the
Hausdorff separation axiom; by $e$ \ we shall denote the unity of
\ $G$.

All topological spaces  are assumed to be Tychonoff (= completely
regular and Hausdorff). The basic ideas and facts of the theory of
$G$-spaces or topological transformation groups can be found in
G.~Bredon \cite{br:72} and in  R.~Palais \cite{pal:60}. Our basic
reference on  proper group actions is Palais' article
\cite{pal:61} (see also \cite{ab:74}, \cite{ab:78}, \cite{elf:01}).

For the  convenience of the reader we recall, however,  some more
special definitions and facts below.

\smallskip

By a $G$-space we mean a topological space $X$ together with a fixed continuous action  $G\times X\to X$ of a
topological  group $G$ on $X$. By $gx$ we shall denote the image of the pair $(g, x)\in G\times X$  under the
action.

 If $Y$ is another $G$-space, a
continuous map $f:X\to Y$ is called a $G$-map or an equivariant
map, if $f(gx)=gf(x)$ for every $x\in X$ and $g\in G$.

If $X$ is a $G$-space, then for a subset $S\subset X$ and for a subgroup $H\subset G$, \ the $H$-hull (or
$H$-saturation) of $S$ is defined as follows:  $H(S)$= $\{hs \ |\ h\in H,\ s\in S\}$. If $S$ is the  one point set
$\{x\}$, then the $G$-hull $G(\{x\})$ usually is denoted by $G(x)$  and called  the orbit of $x$. The orbit
space $X/G$ is always considered in its quotient topology.

A subset $S\subset X$ is called $H$-invariant if it coincides with
its $H$-hull, i.e., $S=H(S)$. By an invariant set we shall mean a
 $G$-invariant set.

For any $x\in X$, the subgroup   $G_x =\{g\in G \ |  \  gx=x\}$ is
called  the stabilizer (or stationary subgroup) at $x$.

 A compatible metric $\rho$ on a metrizable $G$-space $X$  is called invariant or $G$-invariant, if
$\rho(gx, gy)=\rho(x, y)$ for all $g\in G$ and $x, y\in X$. If $\rho$ is a $G$-invariant
metric on any $G$-space  $X$, then it is easy to verify that  the formula
$$\widetilde{\rho}\bigl(G(x), G(y)\bigr)=inf\{\rho(x', y')  \ | \  \ x'\in G(x), \ y'\in G(y)\}$$
 defines  a pseudometric $\widetilde{\rho}$, compatible with the quotient topology of $X/G$. If, in addition, $X$ is a proper $G$-space then  $\widetilde{\rho}$ \ is, in fact,  a metric on $X/G$ \cite[Theorem~4.3.4]{pal:61}.

For a closed subgroup $H \subset G$, by $G/H$ we will denote the $G$-space of cosets $\{gH \mid g\in G\}$ under
the action induced by left translations.

A locally compact group $G$ is called {\it almost connected}, if the space of connected components of
 $G$ is compact. Such a group has a maximal compact subgroup $K$, i.e., every compact subgroup of $G$ is
  conjugate to a subgroup of $K$ (see \cite[Ch.\,H, Theorem~32.5]{stroppel}).

\medskip
\normalfont

In what follows we shall need also  the definition of a twisted product $G\times_K S$, where
 $K$ is a  closed subgroup of $G$,  and   $S$  a $K$-space.
$G\times_KS$ is the orbit space of the $K$-space $G\times S$ on which  $K$ acts by the rule: $k(g,
s)=(gk^{-1}, ks)$. Furthermore, there is a natural action of $G$ on $G\times_K S$ given by
$g^\prime[g, s]=[g^\prime g, s]$, where $g'\in G$ and $[g, s]$ denotes the $K$-orbit of the point
$(g, s)$ in $G\times S$.  We shall identify $S$, by means of the $K$-equivariant embedding $s\mapsto [e, s]$, $s\in S$,  with  the $K$-invariant
subset $\{[e, s] \ | \ s\in S\}$ of $G\times_KS$. This  $K$-equivariant embedding $S\hookrightarrow G\times_KS$ induces a homeomorphism of the $K$-orbit space $S/K$ onto the $G$-orbit space $(G\times_KS)/G$ (see \cite[Ch.~II, Proposition~3.3]{br:72}).

 The twisted products  are  of a particular interest in the theory of
transformation groups  (see \cite[Ch.~II, \S~2]{br:72}). It turns
out that every proper $G$-space  locally is a twisted product. For a more
precise formulation we  need to  recall  the following well-known
notion  of a slice (see \cite[p.~305]{pal:61}):

\begin{definition}\label{D:21} Let $X$ be a $G$-space and   $K$   a closed  subgroup of $G$.
 A  $K$-invariant  subset $S\subset X$ is called a $K$-kernel if  there is a $G$-equivariant map $f:G(S)\to G/K$, {called the slicing map,}
 such that $S$=$f^{-1}(eK)$.   The saturation $G(S)$ is called   a {\it $G$-tubular or just a tubular} set, and the subgroup $K$ will be
 referred to as    the slicing subgroup.

If in addition   $G(S)$ is open in $X$ then  we shall call $S$ a
$K$-slice in $X$.

  If  $G(S)=X$ then  $S$ is called {\it a global} $K$-slice of $X$.
\end{definition}

It turns out that each tubular set $G(S)$ with a compact slicing subgroup $K$ is
$G$-homeomorphic to the twisted product $G\times_K S$; namely the
map  $\xi:G\times_K S\to G(S)$ defined by $\xi([g, s])=gs$ is a
$G$-homeomorphism (see \cite[Ch. II, Theorem 4.2]{br:72}). In what
follows we will  use this fact without a specific reference.

 One of the fundamental  results in  the theory of topological
transformation groups states  (see  \cite[Proposition~
2.3.1]{pal:61}) that, if $X$ is a proper $G$-space with $G$ a Lie
 group, then for any point $x\in X$, there exists a $G_x$-slice $S$ in $X$ such that   $x\in S$.
 In general, when $G$ is not a Lie group, it is no
longer true that a $G_x$-slice exists at each point of $X$ (see
\cite{ant:94} for a discussion). However,  the following approximate version of Palais' slice
theorem  for non-Lie group  actions holds true, which   plays  a key role in the proof of Theorem~\ref{T:13}:

\begin{theorem}[Approximate slice theorem~\cite{ant:large}] \label{T:appr} Let $G$ be a locally compact group,  $X$  a proper $G$-space and  $x\in X$.
Then for any  neighborhood $O$ of $x$ in $X$,   there exist  a    large  subgroup $K\subset G$ with $G_x\subset K$, and a $K$-slice  $S$ such that  $x\in S\subset O$.
 \end{theorem}

 Recall that by a large subgroup here we mean a compact  subgroup $H\subset G$ such that the quotient space  $G/H$ is  a   manifold.

Thus,  every proper $G$-space is covered by invariant open subsets each of which is a twisted product.

\smallskip

 A version of  this theorem, without requiring  $K$ to be a large subgroup  was obtained  earlier in    \cite{ab:78} (see also \cite{ant:94} for the case of compact non-Lie group actions, and \cite{ant:05} for the case of  almost connected groups).   We emphasize  that  namely the property  \lq\lq $K$ is a large subgroup\rq\rq \ is responsible for the applications of Theorem~ \ref{T:appr} in this  paper. We refer the reader to \cite{ant:large} for  discussion of further properties of large subgroups.

\medskip

On any group $G$ one  can define two natural  (but equivalent) actions of $G$ given  by the formulas:
$$g\cdot x=gx,  \quad \quad\text{and}\quad g*x=xg^{-1},$$
respectively, where in the right parts the group operations are used  with $g, x\in G$.

  Throughout we shall use  the second action only.

\

  By $U(G)$ we shall denote the Banach space of all right uniformly continuous bounded functions $f:G\to \Bbb R$ endowed with the supremum norm. Recall that $f$ is called right  uniformly continuous, if for every $\varepsilon >0$ there exists a neighborhood $O$ of the unity in $G$ such that $|f(y)-f(x)|<\varepsilon$ whenever $yx^{-1}\in O$.

  We shall consider the induced action of $G$ on $U(G)$, i.e.,
  $$(gf)(x)=f(xg), \quad\text{for all}\quad  g, x\in G.$$
  It is easy to check that this action is continuous, linear  and isometric (see e.g., \cite[Proposition~7]{ant:87}).

\begin{proposition}\label{P1}  Let $G$ be a topological  group. Then for  every  $f\in U(G)$,  the map   $f_*:G\to U(G)$ defined by $ f_*(x)(g)=f(xg^{-1}),\  x, g\in G $, is a right uniformly continuous $G$-map.
\end{proposition}
\begin{proof} A simple verification.
\end{proof}

The following lemma  of H.~Abels~\cite[Lemma~1.5]{ab:74} is used in the proof of Theorem~\ref{T:12}:

\begin{lemma}\label{P:abels}  Suppose  $G$ is an almost connected group and $K$  a maximal compact subgroup of $G$. Let $X$ be any $G$-space, $A_1$ and $A_2$ two  $G$-subsets of $X$ such that $X= A_1\cup A_2$ and the intersection $A_1\cap A_2$ is closed in $X$,   and let $f_i:A_i\to G/K$, $i=1, 2$,  be $G$-maps. If  the inverse image $f_2^{-1}(eK)$ is a normal space then there exists a $G$-map $f: X \to G/K$ such that $f|_{A_1}=f_1$.
\end{lemma}

In the proof of Theorem~\ref{T:13} the following result due to  K.~Morita~\cite{mor:73}   plays a key role:

\begin{theorem}\label{T:morfil} Let  $X$ be a Tychonoff  space while  $Y$ is a locally compact paracompact space. Then
\begin{equation}\label{morfil}
 {\rm dim}\, X\times Y \le  {\rm dim}\, X + {\rm dim}\, Y,
 \end{equation}
 where  ${\rm dim}$ stands for the covering dimension.
\end{theorem}

One should  mention that the covering dimension  of an arbitrary  Tychonoff space is obtained by a slight modification of  the usual definition of the dimension of a normal space;  namely, one should just  replace  \lq\lq open covering\rq\rq \ by   \lq\lq fuctionally open covering\rq\rq. \
It coincides with the usual definition of covering dimension for normal spaces. This modification is due to Kat\v etov~\cite{kat:50} (see  also  Smirnov~\cite{smir:56}).

\medskip

 An invariant subset $A$ of a  $G$-space $X$ is called  {\it $G$-functionally open} if there exists a  continuous invariant  function   $f:X\to [0, 1]$ such that $A=f^{-1}\big((0, 1]\big)$. Putting  here $G=\{e\}$, the trivial group, we obtain the corresponding notion of a  functionally open set. Clearly, \lq\lq $G$-functionally open\rq\rq \  implies \lq\lq functionally open\rq\rq.

 Recall that in the literature instead of \lq\lq functionally open\rq\rq \   also the term \lq\lq cozero set\rq\rq  \ is used.

It is clear that every open invariant subset $A$ of a  $G$-space $X$ which is metrizable by a $G$-invariant metric is necessarily $G$-functionally open. In fact, if $\rho$ is such a metric on $X$, then the continuous  function $f:X\to [0, 1]$ defined by
$$f(x)= min\{ 1, \rho(x, X\setminus A)\}, \quad x\in X,$$
is   invariant and $A=f^{-1}\big((0, 1]\big)$.

This observation will be used in the proof of Proposition~\ref{P:refined} below.

\medskip

In the proof of Theorem~\ref{T:13} we shall also use the following simple assertion:

\begin{proposition}\label{P:functionally}  Let $G$ be a topological  group, $X$ a $G$-space and $A$ a $G$-functionally open subset of $X$. Then $A/G$ is a functionally open  subset of $X/G$.
\end{proposition}
\begin{proof} A simple verification.
\end{proof}

Finally we recall that by a $\sigma$-discrete cover of a space $X$ we mean a cover which  is the countable union $\bigcup_{n=1}^\infty \mathcal U_n$,  where every family $\mathcal U_n$ is discrete, i.e., for each  point $x\in X$ there is a neighborhood $V$ of $x$ such that $V\cap A\ne \emptyset$ for at most one element $A$ of $\mathcal U_n$.

\smallskip

\section{Proofs}

Recall  that a cover $\omega$ of a $G$-space is called invariant if each member $U\in \omega$ is an invariant set, i.e.,  $G(U)=U$.

A   cover $\{U_i \mid i\in \mathcal I\}$ of a $G$-space $X$ is called  $G$-functionally open  whenever each  $U_i$ is a $G$-functionally open  subset of $X$. Putting here $G=\{e\}$, the trivial group, we obtain the corresponding notion of a functionally open cover.

\begin{proposition}\label{P:refined}  Let  $G$ be a locally compact group,   $X$ a proper $G$-space and $\{U_s\}_{s\in S}$ an invariant  open cover of  $X$. Suppose that  there exists  a $G$-map $f: X\to Z$ to a  (not necessarily proper) $G$-space $Z$  which is  metrizable  by a $G$-invariant metric and has  an  invariant  open cover $\{O_t\}_{t\in T}$ such that the cover $\{f^{-1}(O_t)\}_{t\in T}$ refines $\{U_s\}_{s\in S}$.  Then $\{U_s\}_{s\in S}$ admits  a $G$-functionally open refinement $\{V_r\}_{r\in R}$ which is both locally finite and $\sigma$-discrete.
\end{proposition}
\begin{proof}
Consider the following composition:

$$\begin{array}{cccccccccc}
X &\stackrel{f}{\longrightarrow} &  Z &\stackrel{q}{\longrightarrow} &  Z/G \\

\end{array}
$$
where  $q$ is the $G$-orbit map.

Since $Z$ is metrizable by a $G$-invariant  metric,  the orbit space $Z/G$
is pseudometrizable (see Preliminaries). Hence the open cover
$\{q(O_t) \ | \ t\in T\}$ of $Z/G$ admits a refinement, say $\{W_i\ | \ i\in \mathcal I\}$  which is both locally finite and $\sigma$-discrete (see
\cite[Theorem~4.4.1 and Remark~4.4.2]{eng:77}).

Then, clearly, $\{{q}^{-1}(W_i) \ | \ i\in \mathcal I\}$  \ is an  invariant open refinement of $\{G(O_t) \mid t\in T\}$, and hence,
$\{f^{-1}\big({q}^{-1}(W_i)\big) \ | \ i\in \mathcal I\}$ is an  invariant open refinement of $\{f^{-1}\big(G(O_t)\big) \mid t\in T\}$.
But $f^{-1}\big(G(O_t)\big)=G\big(f^{-1}(O_t)\big)$ since $f$ is a $G$-map. Since  each $f^{-1}(O_t)$ is contained in some $U_s$ and $U_s$ is invariant, we infer that   $G\big(f^{-1}(O_t)\big)$
 is contained in  $U_s$. This yields that the cover $\{f^{-1}\big({q}^{-1}(W_i)\big) \ | \ i\in \mathcal I\}$ is a refinement of $\{U_s \ | \
s\in S\}$.

Further, since $\{W_i\ | \ i\in \mathcal I\}$  is  locally finite and $\sigma$-discrete and $qf$ is continuous we infer that  $\{f^{-1}\big({q}^{-1}(W_i)\big) \ | \ i\in \mathcal I\}$ is also  locally finite and $\sigma$-discrete.

Besides, each $q^{-1}(W_i)$ is $G$-functionally open because it is an invariant open subset of $Z$ which is metrizable by a $G$-invariant metric (see the observation before Proposition~\ref{P:functionally}). Next, since $f$ is a $G$-map, the  inverse image  $f^{-1}\big({q}^{-1}(W_i)\big)$ is also $G$-functionally open.
 Thus,  $\{f^{-1}\big({q}^{-1}(W_i)\big) \ | \ i\in \mathcal I\}$ is the desired refinement of  $\{U_s \ | \
s\in S\}$.

\end{proof}

\bigskip

\noindent
{\it Proof of  Theorem~\ref{T:tub}}.
 By \cite[Theorem~1.1]{ant:10}, $X$ is a proper $G$-space. According to  Theorem~\ref{T:appr},
   there exist a $G$-invariant neighborhood $V$ of the unity in $X$,  a  large subgroup $H\subset G$ and a $G$-equivariant map $\varphi :V\to G/H$ such that $\varphi(e)=eH$.  Due to   regularity of the quotient space $X/G$,  one can choose  a $G$-invariant neighborhood $P$  of the unity such that
$\overline{P}\subset V$.

Choose  a   right  uniformly continuous function  $f:X\to [0, 1]$  such that
 \begin{equation}\label{1}
 f(e)=0 \ \ \  \text{and} \ \ \  f^{-1}\big([0, 1)\big)\subset P.
\end{equation}
Such a function exists because, as is well known,  uniformly continuous bounded functions separate points from  closed sets in any uniform space  (see, e.g.,  \cite[p.\,7, I.13]{isb:64}; for topological groups see  \cite[Theorem~3.3.11]{arhtk}).

Then, by Proposition~\ref{P1}, $f$ induces an  $X$-equivariant map
$f_*:X\to U(X)$ defined by the rule:
$$ f_*(x)(g)=f(xg^{-1}), \ \ \  x, g\in X. $$

Denote by $Z$ the image $f_*(X)$. Clearly, $Z$ is the $X$-orbit of the
point $f_*(e)$ in the $X$-space $U(X)$, and  the metric of $U(X)$ induces an  $X$-invariant
metric on $Z$.

It follows from (\ref{1}) and the $X$-equivariance of $f_*$ that
\begin{equation}\label{2}
f_*^{-1}(\Gamma_{x, Q})\subset x^{-1}P, \ \ \ \text{for every}\ \ \
x\in X,
\end{equation}
where $Q=[0, 1)$ and $\Gamma_{x, Q}=\{\varphi\in Z \mid \varphi(x)\in Q\}$, which is an open subset  of   $Z$.

Besides, since  $f_*(x)\in \Gamma_{x, Q}$  for every $x\in X$, we see that the sets $\Gamma_{x, Q}$, $x\in X$, constitute a cover  of $Z$.

\medskip

In what follows  we restrict ourselves only to  the induced actions of the subgroup $G\subset X$, i.e., we will consider $X$ and $Z$ just as
$G$-spaces. Note that $Z$ may not be  a proper $G$-space and we do not need this fact in the sequel.

\medskip

It follows from (\ref{2}) and from the $G$-equivariance of $f_*$ that
 \begin{equation}\label{3}
 f_*^{-1}\big(G(\Gamma_{x, Q})\big)\subset x^{-1}P, \ \ \  \text{for every} \ \ \   x\in X.
\end{equation}

Thus, the hypotheses of Proposition~\ref{P:refined} are fulfilled for the $G$-map $f_*:X\to Z$ and the $G$-invariant open covers $\{xP \ | \ x\in X\}$ and $\{G(\Gamma_{x, Q}) \ | \ x\in X\}$ of $X$ and $Z$, respectively.
Then by Proposition~\ref{P:refined},  $\{xP\}_{x\in X}$ admits a   $G$-functionally  open refinement $\{W_i\}_{i\in \mathcal I}$ which is both locally finite and $\sigma$-discrete.
Hence, each $W_i$ is contained in some $xP$ which implies that $\overline{W_i}\subset x\overline P\subset xV$.
Now observe that each  set $xV$ is  $G$-tubular; the corresponding slicing $G$-map $\psi:xV\to G/H$ is just the composition
 $$\begin{array}{cccccccccc}
xV &\stackrel{\eta}{\longrightarrow} &  V &\stackrel{\varphi}{\longrightarrow}  & G/H\end{array}
$$
where  $\eta(xv)=v$ for all $v\in V$.
Perhaps it is in order to emphasize here that $V$ and $xV$ are $G$-invariant subsets of $X$ which is equipped with  the action $g*t=tg^{-1}$, where   $g\in G$ and  $t\in X$. Then, clearly $\eta$ is a $G$-map, and hence, the composition $\psi=\varphi \eta$   is also a $G$-map.

Consequently, each   $\overline{W_i}$ being a $G$-invariant subset of  \ $xV$, \ is itself  \ a $G$-tubular set  with the slicing subgroup $H$, as required.

In conclusion, it  remains to observe that  statement  (3) follows immediately from statement (2) (see Section~2, the first paragraph  after Definition~\ref{D:21}).

 \qed

\

\noindent
{\it Proof of Theorem~\ref{T:tubmax}}.
By  Theorem~\ref{T:tub}, $X$ is covered by a locally finite $\sigma$-discrete   family $\{W_i\}_{i\in \mathcal I}$ of $G$-functionally  open  sets $W_i$ such that the closure $\overline{W_i}$ is a  $G$-tubular set associated with a large slicing subgroup $H\subset G$.

    Let  $\psi_i:\overline{W_i}\to G/H$ be the  corresponding slicing map.
 Since, by the maximality of $K$, there exists an evident  $G$-map $\xi:G/H\to G/K$, the composition  $\varphi_i=\xi \psi_i$ \ is a $G$-map  $\varphi_i: \overline{W_i} \to G/K$.
In this case  $\overline{W_i}$ is $G$-homeomorphic to the twisted product $G\times_KS_i$,  where $S_i=\varphi_i^{-1}(eK)$ (see Section~2). Furthermore, by a result of H.~Abels~ \cite[Theorem~2.1]{ab:74},  $\overline{W_i}$ is  homeomorphic (in fact,  $K$-equivariantly homeomorphic) to the cartesian product $G/K\times S_i$, as required.
\qed

\

\noindent
{\it Proof of Theorem~\ref{T:12}}. By Theorem~\ref{T:tub}, cover $X$ by a locally finite collection $\{C_i\}_{i\in I}$ of closed $G$-tubular sets.
  Let  $\psi_i: C_i\to G/H$, $i\in  \mathcal I$, be the corresponding slicing $G$-map, where    $H$  is a   large  subgroup of  $G$.
 Since, by the maximality of $K$, there exists an evident  $G$-map $\xi:G/H\to G/K$, the composition  $\varphi_i=\xi \psi_i$ \ is a $G$-map  $\varphi_i: C_i\to G/K$.

 Well-order the index set $\mathcal I$, and for every $i\in \mathcal I$ put
$$A_i = \bigcup_{j<i}C_j$$
which is closed in $X$ by the local finiteness of the cover $\{C_i\}_{i\in \mathcal I}$.

We aim at constructing inductively a  $G$-map $\phi : X\to G/K$ as follows.

Assume that $i\in \mathcal I$ and the $G$-maps $\phi_j:A_j\to G/K$  are defined  for all $j < i$ in such a way that $\phi_j|_{A_k}=\phi_k$ whenever $k<j$.

Define  $\phi_i:A_i\to G/K$ as follows.
If   $i$ \ is a limit ordinal, then
$$A_i = \bigcup_{j<i}A_j.$$
 Setting  $\phi_i|_{A_j}=\phi_j$ for every $j<i$ we get a well-defined $G$-map  $\phi_i:A_i\to G/K$.   The  continuity of $\phi_i$   follows from the local finiteness of the cover $\{C_i\}_{i\in \mathcal I}$ and from continuity of all  maps   $\phi_j$, $j<i$.

If $i$  is the successor to $j$, then
$$A_i= A_j \cup  C_j.$$
  Observe that each  $C_k$, $k\in \mathcal I$,   is closed in $X$ and hence it is a normal space.   Since  $A_j\cap C_j$ is a closed $G$-invariant subset of $C_j$,  the restriction $\phi_j|_{A_j\cap C_j}$ extends to $C_j$ by Lemma~\ref{P:abels}, and hence, we get a $G$-map   $\phi_i:A_i\to G/K$. This completes the inductive step and the proof. \qed

\bigskip
\noindent
{\it Proof of Corollary~\ref{C:0}}.  Since $G$ is almost connected,  it has a maximal compact subgroup $K$ (see Preliminaries). By   Theorem~\ref{T:12},   $X$ admits a global $K$-slice $S$, and hence, it is $G$-homeomorphic  to the twisted product
 $G\times_KS$ (see Preliminaries). Since the group $K$ is compact, it then follows that the $K$-orbit map
 $$\xi: G\times S\to G\times_KS\cong_GX$$
 is  open and perfect.  This yields immediately that   the restriction  $f= p|_S: S\to X/G$  of the $G$-orbit map $p: X\to X/G$ is an open and perfect surjection. Indeed, it suffices to observe that  for every set $A\subset S$, one has
 $f^{-1}\big(f(A)\big)= KA$, which is open  (respectively, closed  or compact) whenever $A$ is so.
 This completes the proof. \qed

\bigskip

\noindent
{\it Proof of Theorem~\ref{T:13}}. Consider two cases.

\smallskip

\noindent
{\it Case 1.} Assume that $G$ is  connected.  In this case $G$ has a maximal compact subgroup, say, $K$ (see Preliminaries).
Due to  Theorem~\ref{T:tub}, we can cover $X$ by a locally finite collection $\{\Phi_i\}_{i\in \mathcal I}$ of $G$-functionally open tubular sets.
 Let  $\psi_i: \Phi_i\to G/H$ be the corresponding slicing $G$-map, where   the slicing subgroup  $H$  is a (compact)  large  subgroup of  $G$.
 Since, by the maximality of $K$, there exists an evident  $G$-map $\xi:G/H\to G/K$, the composition  $\varphi_i=\xi \psi_i$ \ is a $G$-map  $\varphi_i: \Phi_i \to G/K$.

Since $\{\Phi_i\}_{i\in \mathcal I}$ is a locally finite functionally open cover of $X$,  by virtue of  the locally finite sum theorem of Nagami~\cite[Theorem~2.5]{nag:80}, it suffices to show that ${\rm dim}\, \Phi\le {\rm dim}\, X/G + {\rm dim}\, G$,  for every member $ \Phi$ of the cover $\{\Phi_i\}_{i\in \mathcal I}$.

  Let  $\varphi:\Phi\to G/K$ be the slicing map corresponding to the  tubular set $\Phi\in \{\Phi_i\}_{i\in \mathcal I}$.
In this case  $\Phi$ is $G$-homeomorphic to the twisted product $G\times_KS$,  where $S=\varphi^{-1}(eK)$ (see Preliminaries). Furthermore, by a result of H.~Abels~ \cite[Theorem~2.1]{ab:74},  $\Phi$ is  homeomorphic (in fact,  $K$-equivariantly homeomorphic) to the cartesian product $G/K\times S$.

 Since the group $G$ is locally compact, and hence,  paracompact, we infer that  the quotient space  $G/K$ is also locally compact and paracompact. Then,  according to  Theorem~\ref{T:morfil}, we have:
\begin{equation}\label{eq1}
 {\rm dim}\, \Phi ={\rm dim}\, (G/K\times S)\le  {\rm dim}\, G/K + {\rm dim}\, S.
 \end{equation}

 Since $K$ is compact, according to a result of V.V.~Filippov~\cite{fil:79}, one has the  inequality:
\begin{equation}\label{eq2}
 {\rm dim}\, S\le  {\rm dim}\, S/K + {\rm dim}\, q
 \end{equation}
 where $q: S \to S/K$ is the $K$-orbit projection and ${\rm dim}\, q=\sup\,\{{\rm dim}\, q^{-1}(a) \ | \ a\in S/K\}$.

 Besides,  since $K$ acts freely on $S$, we see that each $K$-orbit $q^{-1}(a)$ is homeomorphic to $K$, and hence,   ${\rm dim}\, q ={\rm dim}\, K.$

Consequently, combining (\ref{eq1}) and (\ref{eq2}), one obtains:
\begin{equation}\label{eq4'}
 {\rm dim}\, \Phi= {\rm dim}\, (G/K\times S)\le  {\rm dim}\, G/K + {\rm dim}\, K +  {\rm dim}\, S/K.
 \end{equation}

Next, since   $\Phi/G\cong (G\times_KS)/G\cong  S/K$ (see Preliminaries) and
$  {\rm dim}\, G / K + {\rm dim}\, K  = {\rm dim}\, G
$
(see Remark~\ref{R}), it then follows from (\ref{eq4'}) that
$$ {\rm dim}\, \Phi\le  {\rm dim}\, \Phi/G + {\rm dim}\, G.$$
Observe that the  quotient $X/G$ is a Tychonoff space. This follows from a  standard fact about proper actions \cite[Proposition~1.2.8]{pal:61} if we remember that $X$ is a proper $G$-space \cite[Theorem~1.1]{ant:10}. It is worth  noting  that a  quotient space $X/M$ is Tychonoff also for any  closed (not necessarily locally compact) subgroup $M\subset X$    (see \cite[Ch.~II, \S~8, (8.14)]{hero:79}).

Further, since  $\Phi/G$ is a functionally open subset of $X/G$ (see Proposition~\ref{P:functionally}), due to monotonicity of the dimension by functionally open  subsets   (see \cite[Theorem~1.1]{nag:80}), one has ${\rm dim}\, \Phi/G  \le {\rm dim}\, X/G$. This,  together with the previous inequality,  yields that
$$ {\rm dim}\, \Phi\le  {\rm dim}\, X/G + {\rm dim}\, G,$$
 as required.

\medskip

\noindent
{\it Case 2}. Let $G$ be  arbitrary locally compact. By case 1,  we have:
\begin{equation}\label{eq6}
 {\rm dim}\, X\le  {\rm dim}\, X/G_0 + {\rm dim}\, G_0,
\end{equation}
where $G_0$ is the unity  component of $G$.

Now, since  ${\rm dim}\, G_0\le {\rm dim}\, G$, it remains to show that
$$
 {\rm dim}\, X/G_0\le  {\rm dim}\, X/G.
$$

Due to  Theorem~\ref{T:tub}, we can cover $X$ by a locally finite collection $\{\Phi_i\}_{i\in \mathcal I}$ of  $G$-functionally open tubular sets.
 Then $\{\Phi_i/G_0\}_{i\in \mathcal I}$ constitutes a functionally open locally finite cover of $X/G_0$ (see Proposition~\ref{P:functionally}).

  Consequently,  by virtue of  the  locally finite sum theorem of Nagami~\cite[Theorem~2.5]{nag:80}, it suffices to show that ${\rm dim}\, \Phi/G_0\le {\rm dim}\, X/G$   for every member $ \Phi/G_0$ of the cover $\{\Phi_i/G_0\}_{i\in \mathcal I}$.

  Let  $\psi:\Phi\to G/H$ be the slicing $G$-map corresponding to the  tubular set $\Phi\in \{\Phi_i\}_{i\in \mathcal I}$,  where  $H$ is a  large subgroup of $G$.
   Then  the quotient $G/H$ is locally connected (in fact, it is a manifold).

    Since the natural map $G/H\to G/G_0H$ is open and the  local connectedness is invariant under  open maps, we infer that $G/G_0H$ is locally connected.  On the other hand, the following natural homeomorphism  holds:
    \begin{equation}\label{isom}
 G/G_0H\cong  \frac{G/G_0}{G_0H/G_0}.
 \end{equation}

Consequently, $G/G_0H$, being the quotient space of the totally disconnected group $G/G_0$, is itself totally disconnected (this follows from  Remark~\ref{R} and   from the fact that in the realm of locally compact spaces total disconnectedness is equivalent to zero-dimensionality~\cite[Theorem~1.4.5]{eng:78}).
 Hence, $G/G_0H$ should be discrete, implying that $G_0H$ is an open subgroup of $G$.

Further, the $H$-slicing $G$-map $\psi:\Phi \to G/H$ induces a $G$-map
$$f: \Phi/G_0\to \frac{G/H}{G_0}=G/G_0H. $$

Since $G/G_0H$ is a discrete group, we infer that  if we put   $S=f^{-1}(eG_0H)$, then each $gS$ is closed and open in $\Phi/G_0$, and  $\Phi/G_0$ is the disjoint union of the  sets  $gS$ one $g$ out of every coset in  $G/G_0H$.  In other words,  $\Phi/G_0$ is just homeomorphic to the product $(G/G_0H)\times S$.

Thus,
$$\Phi/G_0 \cong (G/G_0H)\times S.$$

 Consider  the natural continuous open homomorphism $\pi: G\to G/G_0$ and denote  $L=G_0H/G_0$. Since  $L=\pi(G_0H)=\pi(H)$ while  $G_0H$ is open and $H$ is compact in  $G$, we infer that $L$ is an open compact subgroup of $G/G_0$.

 Next, observe that the composition of  $f$ with  the isomorphism from (\ref{isom}) is just the following   $G$-map:
 $$\Psi: \Phi/G_0\to \frac{G/G_0}{G_0H/G_0}=\frac{G/G_0}{L}.$$
 \smallskip

Since the quotient group $G/G_0$ acts on the spaces  $\Phi/G_0$ and $\frac{G/G_0}{L}$ \ by the induced actions and  since $S=\Psi^{-1}(L)$, we conclude that $S$  is a global  $L$-slice of the $G/G_0$-space $\Phi/G_0$.

This yields (see Preliminaries) that the $G/G_0$-orbit space of $\Phi/G_0$ is just homeomorphic to the $L$-orbit space $S/L$, i.e.,
$$S/L\cong \frac{\Phi/G_0}{G/G_0}.$$
 In its turn, $$\frac{\Phi/G_0}{G/G_0}\cong \Phi/G.$$
So, in sum we get   that $$S/L\cong \Phi/G.$$

Next, since $\Phi/G_0=  (G/G_0H)\times S$ and $G/G_0H$ is discrete, we infer  that
  \begin{equation}\label{last}
\dim \Phi/G_0=\dim S.
  \end{equation}

Further, since $L$ is a compact group, due to the above quoted result of V.V.~Filippov~\cite{fil:79}, we have
   \begin{equation}\label{fil}
\dim S\le \dim S/L+\dim q,
   \end{equation}
 where $q: S \to S/L$ is the $L$-orbit projection and ${\rm dim}\, q=\sup\,\{{\rm dim}\, q^{-1}(a) \ | \ a\in S/L\}$.

 Since $G$ acts freely on $\Phi$ it follows that $G/G_0$ acts freely on $\Phi/G_0$.  But $L$ is a subgroup of $G/G_0$, and hence, its action on $S$ is also free.
This yields  that each $L$-orbit $q^{-1}(a)$ is homeomorphic to $L$, and hence,   ${\rm dim}\, q ={\rm dim}\, L.$

Further, since  $G/G_0$ is totally disconnected, it  is zero-dimensional (see \cite[Theorem~1.4.5]{eng:78}). It then follows from  Remark~\ref{R} that  $\dim\,L =0$.  This, together with (\ref{fil}), implies that
  \begin{equation}\label{final}
\dim\,S \le \dim\,S/L.
   \end{equation}

 Now, taking into account that  $\dim S/L=\dim \Phi/G$,  from (\ref{last}) and (\ref{final}), we get:
 $$\dim\Phi/G_0=\dim S\le \dim S/L= \dim \Phi/G.$$
 But, as we have mentioned above,    $\dim\Phi/G\le  \dim X/G$,  and we finally get the desired inequality $\dim\Phi/G_0\le \dim X/G. $
\qed

\medskip

\subsection*{Acknowledgement}
The author     would like to thank the referee for   useful comments.
This research  was supported in part by grants IN102608 from PAPIIT (UNAM) and  79536 from CONACYT (Mexico).

\medskip

\bibliographystyle{amsplain}

\begin{thebibliography}{1110}

\bibitem{ab:74} H. Abels, {\em  Parallelizability of proper actions,
global $K$-slices and maximal compact subgroups},  Math. Ann. {\bf
212} (1974),  1--19.

\bibitem{ab:78}
H. Abels, {\em  A universal proper $G$-space},  Math. Z. {\bf
159} (1978),  143-158.

\bibitem{ant:87}
S.~A.~Antonian, {\em  Equivariant embeddings into $G$-AR's}, Glasnik
Matemati\v cki {\bf 22 (42)} (1987),  503--533.

\bibitem{ant:94}  S.~A.~Antonyan, {\em  Existence of a slice for arbitrary compact transformation groups},   Matematicheskie Zametki {\bf 56:5}
 (1994)  3-9: English transl. in:  Math. Notes {\bf 56 (5-6)} (1994),  101-1104.



\bibitem{ant:05}
S.~A.~Antonyan, {\em Orbit spaces and unions  of equivariant
absolute neighborhood extensors}, Topology Appl. {\bf 146-147}
(2005), 289-315.

\bibitem{ant:10}
S.~A.~Antonyan, {\em Proper actions on topological groups:\,Applications to quotient spaces}, Proc. AMS, {\bf 138, no. 10} (2010), 3707-3716.


  \bibitem{ant:large}
   S. A. Antonyan, {\em Equivariant extension properties of coset spaces of locally compact groups and approximate slices}, ArXiv preprint no.  	 arXiv:1103.0804v1 [math. GN].


\bibitem{arh:05}
A.~V.~Arhangel'skii, {\em Quotients with respect to locally compact subgroups}, Houston J. Math. {\bf 31, no.~ 1}
(2005), 215-226.

\bibitem{arhtk}
A.~Arhangel'skii and M.~Tkachenko, {\em Topological Groups and Related Structures}, Atlantis Press/World Scientific, Amsterdam-Paris, 2008.

\bibitem{br:72}
 G.~Bredon, {\em  Introduction to compact transformation groups},
 Academic Press,  1972.

\bibitem{elf:01} E. Elfving,  {\em The $G$-homotopy type of proper
locally linear $G$-manifolds. II,} Manuscripta Math.\ {\bf 105} (2001), 235-251.

\bibitem{eng:77}
 R.~Engelking, {\em   General Topology}, PWN-Pol. Sci. Publ.,  Warsaw, 1977.

\bibitem{eng:78}
 R.~Engelking, {\em   Dimension Theory}, PWN-Pol. Sci. Publ.,  Warsaw, 1978.

\bibitem{fil:79} V.~ V.~ Filippov, \emph{Dimensionality of spaces with the action of a bicompact group}, Math. Notes,
  {\bf 25, no. 3} (1979), 171-174.

\bibitem{hoch:65}
 G.~Hochshild, {\em  The structure of Lie groups},  Holden-Day Inc., San Francisco,  1965.

\bibitem{hero:79}
 E.~Hewitt and K.~Ross, {\em Abstract Harmonic Analysis}, V.~I,  Springer-Verlag,  1963.

\bibitem{isb:64}
J. R.~Isbell, {\em   Uniform spaces}, Amer. Math. Soc., Providence, R.I., 1964.

\bibitem{kat:50} M. Kat\v etov, {\em A theorem on the Lebesgue dimension},  \v Casopis P\v est. Mat. Fys. {\bf 75} (1950), 79-87.



\bibitem{mor:73} K. Morita, {\em On the dimension of the product
of Tychonoff spaces},  General Topol. Appl.  3 (1973), 125-133.

\bibitem{nag:80} K. Nagami, {\em Dimension of non-normal spaces}, Fund. Math. {\bf 109, no. 2} (1980), 113-121.


\bibitem{pal:60} R.~Palais, {\em  The classification of $G$-spaces}, Memoirs of the AMS, {\bf 36} (1960).

\bibitem{pal:61} R.~Palais,
{\em  On the existence of slices for actions of non-compact Lie
groups}, Ann. of Math. {\bf  73}, (1961), 295-323.

\bibitem{pas:62}
B.~A.~Pasynkov, {\em On  coincidence of different  definitions of dimension for quotient  spaces of locally bicompact groups}, Uspekhi Mat. Nauk,  {\bf  17,  5(107)} (1962),  129-135.

\bibitem{skl:64} E.~G.~Skljarenko, {\em On the topological structure of locally bicompact groups and their quotient
spaces},  Amer. Math. Soc. Transl., Ser. 2,  {\bf  39}  (1964),
57-82.

\bibitem{smir:56} Yu. M. Smirnov, {\em On the dimension of proximity spaces} (in Russian),  Mat. Sb. {\bf 38} (1956) 283- 302.


\bibitem{stroppel} M. Stroppel,
{\em  Locally compact groups}, European Math. Soc., 2006

\end{thebibliography}

\end{document}